\documentclass[11pt]{amsart}

\usepackage[utf8]{inputenc}
\usepackage{amsmath}
\usepackage{amsthm}

\usepackage{lmodern}
\usepackage{anyfontsize}

\usepackage{biblatex}
\usepackage{amsfonts}
\usepackage{amssymb}

\usepackage{url}
\usepackage{doi}
\usepackage{color}
\usepackage{tikz}
\usepackage{todonotes}
\usepackage{bm}
\usepackage{bbold}
\usepackage{thm-restate} 

\usepackage{subfig}
\usepackage{graphicx}
\usepackage{float}
\usepackage{esdiff}
\usepackage{comment}
\usepackage{mathtools} %:=
\usepackage{xcolor} %colortext
\usepackage{bbm} %1
\usepackage{dsfont} %1
\usepackage{diagbox}
\usepackage{pdfsync}
\newtheorem{thm}{Theorem}
\newtheorem{prop}{Proposition}

\newtheorem{lem}{Lemma}
\newtheorem{conj}{Conjecture}

\newtheorem{fact}{Fact}

\theoremstyle{definition}

\newtheorem{claim}[thm]{Claim}

\newcommand{\Z}{\mathbb{Z}}

\newcommand{\R}{\mathbb{R}}

\newcommand{\PP}{\mathbb{P}}
\newcommand{\E}{\mathbb{E}}

\begin{document}

\title{The Gaussian measure of a convex body controls its maximal covering radius}
\author[M. Szusterman]{Maud Szusterman}
\address[Maud Szusterman]{Department of Mathematics\\ Tel Aviv University}
\email{maud.szusterman@imj-prg.fr}

\thanks{The author was supported by the Israel Science Foundation grant No. 1750/20. In addition, part of this work was conducted while the author was visiting the Hausdorff Research Institute for Mathematics, supported by the Deutsche Forschungsgemeinschaft under the Excellence Strategy (project number EXC-2047/1 – 390685813).}

\begin{abstract}
The well-studied vector balancing constant $\beta(U, V)$ of a pair of convex bodies $(U,V)$, is lower bounded by a lattice counterpart, $\alpha(U,V)$. In \cite{BS97}, Banaszczyk and Szarek proved that $\alpha(B_2^n, V)\leq c$ when $V$ has Gaussian measure at least $\frac{1}{2}$, and conjectured that, for centrally symmetric $V$, $\beta(B_2^n, V)$ is always bounded by a function of the Gaussian measure of $V$, independent of $n$. We resolve this conjecture in the affirmative. Moreover, we show that the analogous result holds for $\alpha(B_2^n, V)$ even without the central symmetry assumption.
\end{abstract}

\maketitle

\section{Introduction}
As part of a study of two-player combinatorial games, Spencer \cite{Spbalgames63} introduced the notion of vector balancing. Given a convex set $V\subset \R^n$ such that $0\in \text{int}(V)$, the gauge function $\|\cdot\|_V$ is defined by $\|x\|_V = \min\{r > 0: x \in rV\}$. For an arbitrary subset $U\subset \R^n$ and $V$ as above, the ``vector balancing constant'' of $U$ with respect to $V$ is defined by:
$$\text{vb}(U,V):=\beta'(U,V):= \sup_{t \in \mathbb N} \hspace{2mm} \max_{u_1, \ldots, u_t \in U} \hspace{2mm} \min_{\epsilon\in \{\pm 1\}^t} \left|\left| \sum_{i=1}^t \epsilon_i u_i\right|\right|_V$$

(The notation $\text{vb}(U,V)$ is prominent in the computer science literature around this topic, while $\beta'$ is used in many papers from the previous century.) In words, $\beta'(U, V)$ is the infimal $r$ such that any finite set of vectors $u_1, \ldots, u_t \in U$ may be ``balanced'', by choosing appropriate signs $\epsilon_i$ for each $u_i$, so that the signed sum lies in $rV$. 

A closely related quantity, also measuring balancing for a pair $(U,V)$, is 
$$\beta(U,V)=\max_{u_1, \ldots, u_n \in U} \min_{\epsilon\in \{\pm 1\}^n} \left|\left| \sum_{i=1}^n \epsilon_i u_i\right|\right|_V.$$ 

The difference between $\beta'$ and $\beta$ is that in the former, one needs to balance arbitrarily large finite subsets of $U$, while in the latter one needs only to balance $n$-tuples of vectors from $U$, where $n$ is the dimension of the ambient space. Assuming that $0 \in U$, it is known that $\beta(U,V)\leq \beta'(U,V) \leq 2\beta(U,V)$ (see for instance \cite{LSV}).

The problem of computing, or estimating $\beta(U, V)$ for specific choices of $U, V$, under general conditions, as well as the problem of efficiently finding a sequence of signs $(\epsilon_i)_{i \le n}$ such that $\|\sum \epsilon_i u_i\|_V$ approximates $\beta(\{u_i\},V)$, has attracted much attention over the years. To state some known results, we introduce a few notations and terminology: let us denote $B_p^n$ the unit $\ell_p$-ball in $\mathbb R^n$, i.e., $B_p^n = \{x \in \mathbb R^n: |x_1|^p + \cdots + |x_n|^p \le 1\}$; in particular, $B_2^n$ is the unit Euclidean ball and $B_\infty^n = [-1, 1]^n$. We denote by $\gamma_n$ the standard gaussian measure on $\mathbb R^n$, i.e., the probability measure with density $d\gamma_n = \frac{1}{(2\pi)^{n/2}} e^{-|x|^2/2}\,dx$.

It is not hard to see that $\beta(B_2^n, B_2^n) = \sqrt n$, or that $\beta(B_{\infty}^n, B_2^n)=\Theta(n)$ (the inclusion $B_{\infty}^n\subset \sqrt{n} B_2^n$ gives an upper bound, while Hadamard matrices give the lower bound). The ``six standard deviations'' theorem of Spencer \cite{S85}, which was independently proven by Gluskin \cite{Gl89}, is that $\beta(B_{\infty}^n, B_{\infty}^n)\leq c\sqrt{n}$ with $c>0$ a universal constant. It is also known that $\beta(B_1^n, B_{\infty}^n)\leq 2$ for all $n\geq 1$ (see Beck and Fiala, \cite{BF81}). The well-known Komlós conjecture asks whether $\beta(B_2^n, B_{\infty}^n)= O(1)$, i.e. whether $\beta(B_2^n, B_{\infty}^n)$ is upper bounded by an absolute constant independent of $n$. 
The best upper bound currently known, due to Banaszczyk \cite{Ban98}, gives
$\beta(B_2^n, B_{\infty}^n) \leq 5\sqrt{2\log n}$ (for all $n\geq 2$). He derives this inequality as a corollary of a more general result, namely that $\beta(B_2^n, V)\leq 5$ for any closed convex set $V\subset \R^n$ such that $\gamma_n(V)\geq \frac{1}{2}$.

The vector balancing problem is closely related to the intensively-studied subject of convex bodies and lattices. Recall that an $n$-lattice is a discrete subgroup of $\mathbb R^n$ which spans $\mathbb R^n$ as a vector space, or, equivalently, a set of the form $L = A\mathbb Z^n$ for $A \in GL_n(\mathbb R)$. We write $\mathcal{K}^n$ for the set of closed convex sets with nonempty interior in $\R^n$; $\mathcal{K}_0^n$ the set of $V\in \mathcal{K}^n$ such that $0\in \text{int}(V)$; and $\mathcal{C}^n$ the set of $V \in \mathcal K^n$ with $V=-V$.

Let $V\in \mathcal{K}^n$, and let $L$ be an $n$-lattice. The covering radius of $V$ with respect to $L$, denoted $\mu(L, V)$ is the least $c>0$ such that $L+cV=\R^n$, where $A + B = \{a + b: a \in A, b \in B\}$ is the Minkowski sum. One has $\mu(L, V) < \infty$ because $V$ has nonempty interior. For instance, one may check that $\mu(\mathbb Z^n, B_p^n)=\frac{n^{1/p}}{2}$, for any $p\in [1, +\infty]$. Note also that $\mu(L, V) = \mu(L, V + x)$ for any $x \in \mathbb R^n$. 
Also of interest in lattice theory are the successive minima, defined as follows: for $U\in \mathcal{K}_0^n$, $\lambda_k(L,U)$ is the least $c>0$ such that $L\cap cU$ contains at least $k$ linearly independent vectors. Here we are primarily interested in the last successive minimum, $\lambda_n$.

In \cite{BS97}, Banaszczyk and Szarek introduced the following quantity: for $U \in \mathcal K_0^n, V \in \mathcal K^n$, set
$$\alpha(U,V):=\sup_{\text{$L$ an $n$-lattice}} \frac{\mu(L,V)}{\lambda_n(L,U)}.$$
For any $U,V$, both $\mu(., V)$ and $\lambda_n(., U)$ are $1$-homogeneous in $L$, so one may assume that $\lambda_n(L, U) = 1$, yielding the following equivalent definition:
$$\alpha(U,V)=\sup \{\mu(L,V) : \text{$L$ an $n$-lattice},\, \text{Span}(L\cap U)=\R^n\}.$$
Denote by $|K|$ the volume of a convex set $K\subset \R^n$. Banaszczyk \cite{Ban93} showed that if $U\in \mathcal{C}^n$ and $V\in \mathcal{K}^n$, one has $\alpha(U,V)\geq \frac{\sqrt{n}}{e\sqrt{2\pi}}\left(\frac{|U|}{|V|}\right)^{1/n}$. (One can check that $\alpha(B_2^n, B_2^n)=\frac{\sqrt{n}}{2}$, hence 
this estimate is the best possible up to some absolute constant.)

Though the definitions of $\alpha$ and $\beta$ seem quite different, these quantities are closely related: it turns out that for any $U \in \mathcal K_0^n$, $V \in \mathcal K^n$ one has $\alpha(U,V)\leq \beta(U,V)$ (see \cite{LSV} or Fact \ref{alphabetaineq} in the appendix). On the other hand, for any $n\geq 2$, there exists $V \in \mathcal K_0^n$ for which the ratio $\frac{\alpha(B_2^n,V)}{\beta(B_2^n, V)}$ is arbitrarily small (take for $V$ a cone with gaussian barycenter far from the origin; see \S \ref{subsec:sym_nec} for more details).

Slightly before the appearance of \cite{Ban98}, Banaszczyk and Szarek \cite{BS97} proved that, if $V\in \mathcal{K}^n$ has measure $\gamma_n(V)\geq \frac{1}{2}$, then  $\alpha(B_2^n, V)\leq c$, where $c>0$ is such that $\gamma_1([-c,c])=\frac{1}{2}$. This is sharp in every dimension: the bound is attained by symmetric slabs of gaussian measure $\frac{1}{2}$. 

To the best of our knowledge, it remains unknown whether or not, for $n\geq 2$, there exists $c_n>0$ such that $\alpha(B_2^n, V) \geq c_n \beta(B_2^n, V)$, for any $V\in \mathcal{C}_n$. If true, and if $c_n$ could be chosen independently of $n$, then the Komlós conjecture would be equivalent to the (a priori) weaker conjecture that $\alpha(B_2^n, B_{\infty}^n)=O_n(1)$.

Denote $\Psi(x)=\gamma_1([-x,x])$. The main result obtained by Banaszczyk and Szarek in \cite{BS97} is:
\begin{thm}
\label{BSresult}
Let $V$ be a convex body in $\R^n$, such that $\gamma_n(V)\geq \frac{1}{2}$. Then $\alpha(B_2^n, V)\leq (2\psi^{-1}(1/2))^{-1}$.
\end{thm}

In that paper, the two authors conjectured the following:
\begin{conj}
\label{BSconj}
There exists a non-increasing function $f:(0,1)\to (0,+\infty)$, such that for any symmetric convex body $V\subset \R^n$, $\beta(B_2,V)\leq f(\gamma_n(V))$.
\end{conj}

It is natural to also consider a weaker version of this conjecture, replacing $\beta$ by $\alpha$:
\begin{conj}
\label{weakBSconj}
There exists a non-increasing function $f:(0,1)\to (0,+\infty)$, such that for any symmetric convex body $V\subset \R^n$, $\alpha(B_2, V)\leq f(\gamma_n(V))$.
\end{conj}

Our first result (Theorem \ref{thm:BSconjholds}) is that Conjectures \ref{BSconj} and \ref{weakBSconj} are true. In fact, they follow easily from the $S$-inequality of Latała and Oleszkiewicz (\cite{LO}), together with the results of Banaszczyk \cite{Ban98} and Banaszczyk-Szarek \cite{BS97}, respectively, which were proven shortly beforehand. We prove this in \S \ref{section:BSconj}.

For each of the above two conjectures, one may ask whether the symmetry assumption (on $V$) is needed. Our main result is that the ``$\alpha$-version'' of the conjecture, namely Conjecture \ref{weakBSconj}, holds in greater generality. More precisely, we prove in \S\S \ref{section:main}-\ref{section:conescomputing} the following generalization of Theorem \ref{BSresult}:

\begin{restatable}{thm}{mainthm}
\label{thm:main}
Let $n\geq 1$, and let $K\subset \R^n$ be a closed convex set with gaussian measure $\gamma_n(K) > 0$. Then
$\alpha(B_2^n ,K)\leq f(\gamma_n(K))$, with 
$$f(p)= \begin{cases}
    (2\Psi^{-1}(p))^{-1} & p\in (0, \frac{1}{2}] \\ 
    (2\Psi^{-1}(\frac{1}{2}))^{-1} & p \in (\frac{1}{2}, 1)
\end{cases}.$$
\end{restatable}

It follows from Theorem \ref{thm:main} and from translation invariance (in $K$) of the covering radius, that
$$\alpha(B_2^n, K) = \inf_{x \in \mathbb R^n} \alpha(B_2^n, K+x) \leq \inf_{x \in \mathbb R^n} f(\gamma_n(K+x)),$$
while an analogous statement for $\beta$ is open (but holds at least in case $K=-K$), see \S \ref{section:open}.

In \S \ref{subsec:smallp}, we explain why our results do not yield any improvement to the estimate $\alpha(B_2^n, B_{\infty}^n)\leq C\sqrt{\log n}$ (implied by Theorem \ref{BSresult}), and why the $f$ solving Conjecture \ref{BSconj} (yielded by Banaszczyk's result and by the $S$-inequality, see \S \ref{section:BSconj}) doesn't improve the upper bound $\beta(B_2^n, B_{\infty}^n)\leq C\sqrt{\log n}$ (\cite{Ban98}).

One might wonder whether a similar statement to Theorem \ref{thm:main} holds for $\beta$ rather than $\alpha$. In \S \ref{subsec:sym_nec}, we provide a counter-example which shows that the symmetry assumption cannot be completely removed in Conjecture \ref{BSconj}. We leave open whether $\mathcal{C}^n$ is the largest subset of $\mathcal{K}^n$ for which a dimensionless upper bound (as in Conjecture \ref{BSconj}) holds (see question (2) in \S \ref{section:open}).

In \S \ref{section:BSconj}, we show how to derive Banaszczyk-Szarek's conjecture from \cite{Ban98} and \cite{LO}; and we argue that the $\beta$-version of Theorem \ref{thm:main} fails. In \S \ref{section:main}, we prove Theorem \ref{thm:main} modulo the technical Proposition \ref{prop:planar}, whose proof is carried out in \S \ref{section:conescomputing}.
Finally, \S \ref{section:remarks} collects some remarks on $\alpha(B_2^n, V)$ and $\beta(B_2^n, V)$ for $\gamma_n(V)$ close to $0$ or to $1$, discusses the possibility of improving the upper estimates on $\alpha(B_2^n, B_\infty^n)$ and $\beta(B_2^n, B_\infty^n)$ using Theorems \ref{thm:main} and \ref{thm:BSconjholds}, and concludes with some open questions. 

\textbf{Acknowledgments. } The author is indebted to Fedor Nazarov for the proof of Lemma \ref{lem:red2}. She would like to thank Sander Gribling for providing her with simulations of the function $m(\theta)$, as well as Matthieu Fradelizi, Jacopo Ulivelli, Beatrice-Helen Vritsiou, Artem Zvavitch, for helpful discussions, and Eli Putterman for thoroughly reading the manuscript and suggesting many improvements.

\section{Proof of Banaszczyk-Szarek Conjecture \ref{BSconj}}
\label{section:BSconj}

Let $\mathcal{K}_0^n$ denote the set of convex bodies in $\R^n$ which contain the origin in their interior. For $V\in \mathcal{K}_0^n$ and $U$ an arbitrary subset of $\R^n$, recall that the vector balancing constant $\beta(U,V)$ is defined as the least $C>0$ such that for any $n$-tuple $(u_1, \ldots, u_n)\in U^n$, one can find signs $\epsilon_i=\pm 1$ such that $\sum_{i=1}^n \epsilon_i u_i \in CV$. Clearly, for any $a > 0$ and any $U, V$, $\beta(aU, V) = a \beta(U, V)$ and $\beta(U, aV) = \frac{1}{a} \beta(U, V)$. 

In \cite{Ban98}, Banaszczyk proved that if $\gamma_n(K)\geq \frac{1}{2}$, and $u_1, \ldots , u_t$ is a sequence of vectors with $||u_i||_2\leq 1$, then there exist signs $\epsilon_i=\pm 1$ such that $\sum_{i=1}^t \epsilon_i u_i \in 5K$. In other words:
\begin{thm}[\cite{Ban98}]
\label{Banresult}
Let $K$ be a convex body such that $\gamma_n(K)\geq \frac{1}{2}$. Then $\beta(B_2^n, K)\leq 5$.
\end{thm}

The $S$-inequality \cite{LO} states that the gaussian measure of a symmetric convex set $V$ increases under dilation by a factor $t > 1$ at least as fast as that of a symmetric slab $S$ which satisfies $\gamma_n(S)=\gamma_n(V)$. More precisely:
\begin{thm}[\cite{LO}]
\label{LOineq}
Let $V\in \mathcal{C}_n$ and let $I = [-c,c] \subset \mathbb R$ be a symmetric interval such that $\gamma_n(V) = \gamma_1(I)$. Then for any $t\geq 1$,$\gamma_n(tV)\geq \gamma_1(tI)$.
\end{thm}

Banaszczyk and Szarek's conjecture follows easily from the above two inequalities. Let $\psi: [0, \infty) \to [0, 1)$ be defined by $\psi(u) = \gamma_1([-u,u])$.
\begin{thm}
\label{thm:BSconjholds}
Let $f_{\beta}: [0, 1] \to \mathbb R^+$ be the function defined by 
\begin{equation}f_{\beta}(p) = \begin{cases} 
5\frac{\psi^{-1}(1/2)}{\psi^{-1}(p)} &  p\in (0,\frac{1}{2}] \\
5 & p\in (1/2, 1). 
\end{cases}
\end{equation}
Then, for any symmetric convex body $K$ in $\R^n$, $\beta(B_2^n, V)\leq f_{\beta}(\gamma_n(V))$.
\end{thm}
\begin{proof}
It follows from Theorem \ref{Banresult} that if $\gamma_n(V) = p\geq \frac{1}{2}$, then $\beta(B_2^n, V)\leq 5=f(p)$. Hence we may assume that $\gamma_n(V)=p<1/2$. Since $V$ is symmetric, the $S$-inequality (Theorem \ref{LOineq}) yields $\gamma_n(tV)\geq \frac{1}{2}$ for $t=\frac{\psi^{-1}(1/2)}{\psi^{-1}(p)} > 1$. Therefore, applying Theorem \ref{Banresult} to $tK$ and using the homogeneity of $\beta$, one gets:
$$\beta(B_2^n, V)=t\beta(B_2^n, tV)\leq 5t=f(p).$$
\end{proof}
Since $\alpha(U,V)\leq \beta(U,V)$ for any pair $(U,V)\in (\mathcal{C}^n)^2$, Theorem \ref{thm:BSconjholds} implies in particular that Conjecture \ref{weakBSconj} holds, with the same $f$. However, directly using Theorem \ref{BSresult}, and the $S$-inequality, yields a sharper upper bound, namely one gets that, for any $n\geq 1$, and any symmetric convex body $K\subset \R^n$,  $\alpha(B_2^n ,K)\leq f_\alpha(\gamma_n(K))$, with 
\begin{equation}\label{eq:f_alpha} f_\alpha(p)= 
\begin{cases} (2\Psi^{-1}(p))^{-1} & p\leq \frac{1}{2} 
\\ f_\alpha(p) = (2\Psi^{-1}(\frac{1}{2}))^{-1} & \frac{1}{2}< p < 1.
\end{cases}
\end{equation}
This latter upper bound on $\alpha$ is smaller than the one implied by Theorem \ref{thm:BSconjholds}, by a factor $10 \Psi^{-1}(1/2)$. 

Moreover, the bound $\alpha(B_2^n,V)\leq f_{\alpha}(\gamma_n(V))$ is sharp for $p\in(0,\frac{1}{2}]$, as $\alpha(B_2^n,S)=(2\Psi^{-1}(p))^{-1}$, if $S$ is a symmetric slab of gaussian measure $p$. Indeed, for any $c>0$, if $S=\R^{n-1} \times [-c,c]$, then $\mu(\Z^n,S)=\frac{1}{2c}$, while $\lambda_n(\Z^n, B_2^n)=1$, hence $\alpha(B_2^n, S)\geq \mu(\Z^n,S)= \frac{1}{2c}$. On the other hand, it is easily checked that if $u_1, \ldots, u_n \in B_2^n$ are linearly independent, and $L= \Z u_1 + \cdots + \Z u_n$, then $L+\frac{1}{2c}S=\R^n$. Therefore $\alpha(B_2^n, S)=\frac{1}{2c}$. Since $\gamma_n(S)=\Psi(c)$, this shows sharpness in (\ref{eq:f_alpha}), for $p\leq \frac{1}{2}$.

\subsection{Necessity of the symmetry assumption}\label{subsec:sym_nec}

In view of Theorem \ref{thm:main}, one may wonder if the symmetry assumption is necessary, or if one could weaken the assumption on $K$, and only require $K$ to be a convex body containing the origin in its interior. To see why the symmetry assumption is necessary, consider a cone $$C=\text{Cone}(0, e_n+t B_2^{n-1})=\cup_{x\in B_2^{n-1}} \R_{\geq 0} (e_n+tx)$$ 
Intuitively, one has $\alpha(B_2^n, C)=0$, while $\beta(B_2^n, C)= +\infty$. In fact $\beta(B_2^n, C)$ is not well-defined, so we approximate $C$ with a convex body $V$ containing $0$ in its interior, and with $\beta(B_2^n, V)$ arbitrarily large, which still satisfies $\alpha(B_2^n, V)\leq 1$ and $\gamma_n(V)\geq p$ with $p>0$ fixed. Let us give more detail on how we choose $V$.

Fix $p<\frac{1}{2}$. Let $B=B_2^n \cap e_n^{\perp}$, and, for $d, t>0$, let $C_{d,t}=\text{Conv}(0, d (e_n+ t B))$; that is, $C_{d, t}$ is the cone on an $(n - 1)$-dimensional ball of radius $dt$, with height $d$ and apex at the origin. Denote by $C_{\infty,t}$, or simply by $C_t$, the infinite cone $C_t=\text{Cone}(0, e_n + tB)=\cup_{x\in B} \R_{\geq 0} (e_1+tx)$. Let $t_p>0$ be such that $\gamma_n(C_{t_p})=p$. For any $t>t_p$, there exists $d<\infty$ such that $\gamma_n(C_{d,t})> p$ and $\alpha(B_2^n, C_{d,t})\leq 1$. Indeed, when $d$ is large enough (depending on $t$), $C_{d,t}$ contains a translate of $n B_2^n$, and hence contains a translate of any parallelogram $P=[0,u_1]+ \cdots +[0, u_n]$ with $\|u_i\|_2\leq 1$, ensuring that $\alpha(B_2^n, C_{d,t})\leq 1$.

Let $t>t_p$ and let $d$ be sufficiently large. Since $\text{dist}(e_i, C_{d,t}) > 0$ for $i \neq n$, one may fix $u_2, \ldots, u_{n} \in e_1^{\perp}$ linearly independent, with $\|u_i\|_2 \leq 1$, in a way such that $C_{d,t}\cap P^0=\emptyset$, where $P^0$ is the discrete set: $P^0=\{\epsilon_1 e_1 +\sum_{i=2}^{n} \epsilon_i u_i, \epsilon\in \{\pm 1\}^n\}.$
Denote $\delta:=\text{dist}(P^0, C_{d,t})>0$.
For $s>0$, denote $C'_s=C_{d,t}-se_n$, which contains the origin in its interior. Then, for small $s>0$ (say $s<\frac{\delta}{2}$), one has $\text{dist}(P^0, C'_s)\geq \delta - s >0$, so that $P^0\cap C'_s=\emptyset$. Moreover one can check that $\beta(\{e_1, u_2, \ldots, u_n\}, C'_s)\geq \frac{\delta}{s}$, implying $\beta(B_2^n, C'_s)\geq \frac{\delta}{s}$. Thus, if $s$ is small enough (with respect to $\gamma_n(C_{d,t})-p>0$), one has
$$\gamma_n(C'_s)\geq \gamma_n(C_{d,t})- \gamma_n(C_{d,t} \cap \{x_n \leq s\})\geq \gamma_n(C_{d,t})- \frac{s}{2} >p,$$
 while $\beta(B_2^n, C'_s)$ can be arbitrarily large (by choosing $s>0$ very small), this tells us that no inequality of type $\beta(B_2^n, V)\leq f(\gamma_n(V))$ could hold for arbitrary $V\in \mathcal{K}_0^n$. In other words, it is possible that Theorem \ref{thm:BSconjholds} holds in a more general setting than the class of symmetric convex bodies $\mathcal{C}_n$ (see question (3) in section \ref{section:open}), but it cannot hold for all of $\mathcal{K}_0^n$.

\section{Proof of Theorem \ref{thm:main}}
\label{section:main}

In this section, we state and prove the main result of this article, which gives a positive answer to Conjecture \ref{weakBSconj} in the non-symmetric setting: 

\mainthm*

The proof uses the following technical proposition about the Gaussian measure in $\R^2$.

\begin{restatable}{prop}{planar}
\label{prop:planar}
Fix $p\leq \frac{1}{2}$. Let $W=\{(x,y)\in \R^2 : x\leq t_y\}$ for some concave function $(t_y)$. Assume that $\|W\cap \Delta_p\|_2<2\Psi^{-1}(p)$, where $\|I\|$ denotes the Euclidean length of an interval $I$. Then $\gamma_2(W)<p$.
\end{restatable}

To motivate Proposition \ref{prop:planar}, note that it is reminiscent of (a special case of) Lemma 1 in \cite{BS97}, namely the fact that if $W\subset \R^2$ is convex and $\gamma_2(W)\geq \frac{1}{2}$, then $||W\cap u^{\perp}||_2 \geq c:=2\Psi^{-1}(1/2)$ for any $u\in \mathbb{S}^1$. This lemma is crucial within the proof of the main result of \cite{BS97}.

No such statement can hold for $p<\frac{1}{2}$, since there exist (non-symmetric) closed convex sets $W\subset \R^n$, such that $\gamma_n(W)=p$, and $W\cap H = \emptyset$, with $H = e_n^{\perp}$. Nonetheless, Proposition \ref{prop:planar} suffices to carry out the core step of the proof in this more general setting. 

We postpone the proof of Proposition \ref{prop:planar} to the next section.

\begin{proof}[Proof of Theorem \ref{thm:main} assuming Proposition \ref{prop:planar}]

Since $\mu(L, V + x) = \mu(L, V)$ for any $n$-lattice $L$ and any $x \in \mathbb R^n$, it follows that for any $U, V \in \mathcal{K}^n$, one has $\alpha(U,V)=\alpha(U,x+V)$.  Therefore, for $n=1$, it is easy to see that $\alpha([-1,1], K)=0$ if $K$ is unbounded, while $\alpha([-1,1], [x,x+2c])=(2c)^{-1}$ if $K$ is a closed interval of length $2c>0$. In this latter case, since $\gamma_1([-c_p,c_p]):=\gamma_1(K)\leq \gamma_1([-c,c])$, one deduces
$$\alpha([-1,1], [x,x+2c])=(2c)^{-1} \leq (2c_p)^{-1}=(2\Psi^{-1}(\gamma_1(K)))^{-1},$$
as desired.

Now let $n\geq 2$ and assume that Theorem \ref{thm:main} has been proven in dimension $n-1$. Fix $K$ a closed convex set in $\R^n$, with $p:=\gamma_n(K)>0$. We assume $p<\frac{1}{2}$, since otherwise the
inequality follows from Theorem \ref{BSresult}. 

Assume for the sake of contradiction that $\alpha(B_2^n,K)>f(p)=\frac{1}{2c_p}$, with $c_p:=\Psi^{-1}(p)$. By definition of $\alpha$, there exists an $n$-lattice $L$ with $\lambda_n(2c_p B_2^n, L)\leq 1$ and $\mu(L,K)>1$. This means that there exist $u_1, \ldots, u_n \in (2c_p)B_2^n \cap L$ spanning $\R^n$, and $a\in \R^n$ such that $(a+L)\cap K=\emptyset$.

Fix $(e_1, \ldots, e_n)$ an orthonormal basis of $\R^n$, such that $e_n^{\perp}=\text{Span}(u_1, \ldots, u_{n-1})$. Hence $L':= e_n^{\perp} \cap L$, is an $(n-1)$-lattice satisfying $\lambda_{n-1}(2c_p B_2^{n-1}, L')\leq 1$, where we identify $\R^{n-1}$ with $e_n^{\perp}$ and denote $B_2^{n-1}:=B_2^n \cap e_n^{\perp}$. Denote by $\pi_n$ the orthogonal projection from $\mathbb R^n$ onto $\R e_n$. Since $\pi_n(u_n)\neq 0$, we see that $\pi_n(L)\supseteq \lambda \Z e_n$, where $\lambda:=|\langle u_n, e_n\rangle | \le 2c_p$. Hence $\pi_n(a+L)$ contains a coset $a_n + \lambda \Z$, where we write $a = (a', a_n)$ with $a' \in e_n^\perp, a_n \in \mathbb R$.

For $z\in \R$, denote $K_z:=(K-ze_n) \cap e_n^{\perp}$, the $(n-1)$-dimensional section of $K$ at height $z$ (or rather, its projection onto $e_n^{\perp}$). By taking intersection with the affine halfplane $H_{a_n}=a+e_n^{\perp}=a_n e_n +e_n^{\perp}$, one sees that the assumption $(a+L)\cap K=\emptyset$ yields $(a'+L')\cap K_{a_n}=\emptyset$. Indeed
$$(a+L) \cap H_{a_n}=
(a+ L) \cap (a +e_n^{\perp})=a+(L \cap e_n^{\perp})=a+L'=a_n e_n +(a'+L')$$ 
$$\text{ while} \hspace{4mm} K \cap H_{a_n} = K \cap (a_n e_n+e_n^{\perp}) = a_n e_n + \left((K-a_ne_n)\cap e_n^{\perp}\right)=a_n e_n + K_{a_n} $$
so that $\emptyset=(a+L)\cap K=\left((a+L) \cap H_{a_n} \right) \cap \left(K \cap H_{a_n} \right)= a_n e_n + \left( (a'+L') \cap K_{a_n}\right)$.

Thus, we have shown that $\mu(K_{a_n}, L')>1$. Since $L'$ satisfies $\lambda_{n-1}(2c_p B_2^{n-1}, L') \leq 1$, the inductive assumption gives us $\gamma_{n-1}(K_{a_n})<p$. 

Now consider an arbitrary $b \in a + L$ in the coset of $a$. Since $b+L=a+L$ for any such $b$, the same argument as above yields that, if one writes $b=(b',b_n)$ (with $b'\in e_n^{\perp}$ as before), $\gamma_{n-1}(K_{b_n})<p$ for any $b \in a+L$.
Since $\pi_n(L)\supseteq \lambda \Z e_n$, with $\lambda:=|\langle u_n, e_n\rangle | \in (0, 2c_p]$, we deduce that $\gamma_{n-1}(K_z) <p$ for any $z\in a_n +\lambda \Z$.

Let $W\subset \R^2$ be the two-dimensional Ehrhard symmetrization of $K$: that is, $W=\cup_{z\in \R} (W_z+ z e_2)$ where $W_z=(-\infty, t_z] e_1$ and $t_z=\Phi^{-1}(\gamma_{n-1}(K_z))$ (i.e., $t_z$ is chosen so that $\gamma_1(W_z)=\gamma_{n-1}(K_z)$). Then $W$ is closed and convex, by the $(n-1)$-dimensional Ehrhard inequality, and $\gamma_2(W)=\gamma_n(K)=p$. The fact that $\gamma_{n-1}(K_z) <p$ for any $z\in a_n +\lambda \Z$, tells us that $t_z<-w:=\Phi^{-1}(p)$ for any $z\in a_n +\lambda \Z$. By convexity of $W$, this implies that $W\cap \Delta_p$ is a segment of length strictly less than $\lambda$, where $\Delta_p$ is the vertical line $\Delta_p=-w e_1+\R e_2$. Therefore $||W \cap \Delta_p||_2 < \lambda \leq 2 c_p=2\Psi^{-1}(p)$, which yields $\gamma_2(W)<p$ by Proposition \ref{prop:planar}, contradiction.
\end{proof}

It thus only remains to prove Proposition \ref{prop:planar}. In the next section, we reduce Proposition \ref{prop:planar} to a planar isoperimetric question involving a one-parameter family of cones (see Lemma \ref{lem:red2}). Then, we solve this isoperimetric question by elementary methods, completing the proof of Proposition \ref{prop:planar}, and hence also the proof of Theorem \ref{thm:main}.

\section{Proof of Proposition \ref{prop:planar}}
\label{section:conescomputing}

Fix $p < \frac{1}{2}$. We denote $-w:=\Phi^{-1}(p)$, and $\Delta_p$ the vertical line $\Delta_p:=-w e_1 + \R e_2$. For any closed interval $A\subset \R^2$, we denote $||A||_2 \in [0, \infty ]$ its euclidean length. We remind the notation $\Psi(u)=\gamma_1([-u,u])$. Recall that our goal is to prove the following proposition:

\planar*

If $p\leq \frac{1}{2}$, denote by $\mathcal{C}_2(p)$ the set of cones $C\subset \R^2$, such that $C \cap (ae_1+\R e_2)=\emptyset$ if $a>0$ is large enough (which we call ``cones infinite on the left'') and such that $||C \cap \Delta_p||_2 \leq 2\Psi^{-1}(p)$. We claim that Proposition \ref{prop:planar} reduces to the following Lemma \ref{lem:red1}: 
\begin{lem}
\label{lem:red1}
Fi $p\leq \frac{1}{2}$. If $C\in \mathcal{C}_2(p)$, then $\gamma_2(C)<p$.
\end{lem}
To show the reduction, it suffices to argue that any $W$ as in Proposition \ref{prop:planar} is contained in a cone $C$ as in Lemma \ref{lem:red1}. This is almost true, up to two exceptions, which roughly speaking are the symmetric slab $S_p:=\{(x,y) : |y|\leq \Psi^{-1}(p)\}$, and the halfplane $H_p =\{(x,y) : x\leq -w\}$, for which it is easy to see directly that the desired result holds. 

\begin{proof}[Proof that Lemma \ref{lem:red1} implies Proposition \ref{prop:planar}]
If $W\cap \Delta_p=\emptyset$, then, since $W$ is a ``horizontal hypograph'', we have $W\subset H_p$, and moreover $(-w,0)\notin W$. Hence $D\cap W=\emptyset$ for some small disk around $(-w,0)$. One deduces that $\gamma_2(W) < \gamma_2(H_p)=p$. Similarly, if $W\cap \Delta_p$ is a singleton, then $W\subset H_p$ and $W\cap D=\emptyset$ for some small disk centered somewhere on $\Delta_p$, yielding again $\gamma_2(W)<p$.

Hence we may assume that $W\cap \Delta_p$ is a segment $[a,b]$; more formally, we can assume that $W\cap \Delta_p=-w e_1+[a,b] e_2$, for some $a<b$. Let $T_a$ and $T_b$ be two lines tangent to $W$, respectively at $Q_a=(-w,a)$ and at $Q_b=(-w,b)$. Since $W$ is a (horizontal) hypograph, $T_a$ must have non-negative slope, and $T_b$ non-positive slope. If both $T_a$ and $T_b$ are horizontal, then $W \subset \{(x,y) : a\leq y\leq b\}$, so $\gamma_2(W)\leq \gamma_1([a,b])<p$, where the last inequality follows because $b - a < 2\Psi^{-1}(p)$. So we may assume one of the two slopes is non-zero. If $T_a$ or $T_b$ is vertical, then $W\subset H_p$, and since $W\cap \Delta_p \neq \Delta_p$, we conclude as above that $\gamma_2(W)<p$. Hence we can assume neither of $T_a, T_b$ is vertical.
Denote by $H_a$ the halfplane above $T_a$, and $H_b$ the halfplane below $T_b$. Then, since $T_a$ and $T_b$ are tangent lines to $W$, we see that $W \subset C:= H_a \cap H_b$, and, since at least one of $T_a, T_b$ is non-horizontal, and both are non-vertical, we see that $C\in \mathcal{C}_2(p)$.
\end{proof}
Note that, since $\gamma_2$ has strictly positive density, Lemma \ref{lem:red1} reduces to considering those $C\in \mathcal{C}_2(p)$ such that the segment $C\cap \Delta_p$ has length $2\Psi^{-1}(p)$ exactly. Next, we reduce to considering cones $C\in \mathcal{C}_2(p)$ which are symmetric with respect to the $x$-axis. 
\begin{lem}
\label{lem:red2}
Let $C$ be a cone ``infinite on the left,'' symmetric with respect to $x$-axis, and such that the segment $C\cap \Delta_p$ has (euclidean) length $2 \Psi^{-1}(p)$. Then $\gamma_2(C)<p$.
\end{lem}

\begin{proof}
[Proof that Lemma \ref{lem:red2} implies Lemma \ref{lem:red1}]
Let $C$ be a cone as in Lemma \ref{lem:red1}. Consider the cone $C'$ which is the Steiner symmetrization of the cone $C$ with respect to the $x$-axis: for every $x \in \mathbb R$, let $\ell_x = \|C \cap (xe_1 + \mathbb Re_2)\|_2$, and let 
$$C' = \left\{(x, y) \in \mathbb R^2:  (xe_1 + \mathbb Re_2) \cap C \neq \emptyset, |y| \leq \frac{l_x}{2} \right\}.$$

That is, $C'$ is defined by replacing each vertical slice of $C$ by a vertical segment of the same length centered on the $x$-axis (see Figure \ref{fig:steiner_cone}). Note that this does not decrease the Gaussian measure of the slice, since $\gamma([a, b]) \le \gamma([-\frac{b - a}{2}, \frac{b - a}{2}])$ for any $a \le b$.

It is easily verified that $C'$ is a cone, with $||C'\cap \Delta_p||_2=||C\cap \Delta_p||_2=l_{-w}$. Moreover,
$$\gamma_2(C)=\int_{-\infty}^\infty \gamma_1((C - x e_1) \cap \mathbb Re_2) \,d\gamma_1(x) \leq \int_{-\infty}^\infty \gamma_1((C' - x e_1) \cap \mathbb Re_2) \,d\gamma_1(x) = \gamma_2(C').$$

Assuming Lemma \ref{lem:red2}, we have $\gamma_2(C') < p$, which then implies $\gamma_2(C) < p$, as desired.
\end{proof}

\begin{figure}[ht]
\centering
\includegraphics[width=0.7 \textwidth]{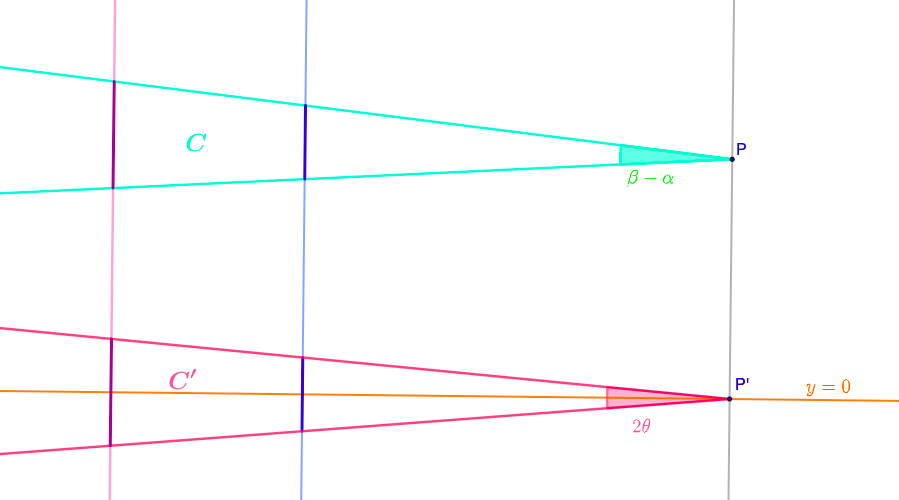}
\caption{Symmetrizing a 2D cone}\label{fig:steiner_cone}
\end{figure}

Therefore, to prove Proposition \ref{prop:planar}, it is enough to prove Lemma \ref{lem:red2}. The next subsection is dedicated to proving this Lemma. We are indebted to F. Nazarov for providing us a complete proof, which is partly reproduced in the next subsection (in particular Lemma \ref{nicelemma} and Claims \ref{claimeasy},\ref{penible}).

\subsection{Proof of Lemma \ref{lem:red2}}
\label{section:Nazarovandmore}
Fix $p\in (0,\frac{1}{2})$. Let  $w= -\Phi^{-1}(p)$ as above, and also let $h=\Psi^{-1}(p)>0$, i.e., $w$ and $h$ are chosen so that $\gamma_1((-\infty,-w]) = \gamma_1([-h,h])=p$. Let $P=(-w,h)$ and $P'=(-w,-h)$.
 Occasionally, we may denote $h=h_p$ and $w=w_p$.

Let $H_+=\{(x,y)\in \R^2 : y\geq 0\}$ be the upper half-plane. Fix some $y>-w$, let $Y=(y,0)\in \R^2$, and denote by $C$ the symmetric convex cone bounded by the two rays $\overrightarrow{YP}$ and $\overrightarrow{YP'}$. Let $2\theta$ be the angle of $C$, i.e. $\tan(\theta)=\frac{h}{w+y}$. It will be more convenient for us to think of $C$ as a function of $\theta$, rather than of $y$; as $\theta$ runs from $0$ to $\frac{\pi}{2}$, $C_\theta$ runs over all the cones considered in Lemma \ref{lem:red2}. Hence, denoting $m(\theta)=\gamma_2(C_{\theta}\cap H_+)=: \gamma_2(C_{\theta}^+)$, our goal is simply to prove that $m(\theta) < \frac{p}{2}$ for all $\theta \in (0, \frac{\pi}{2})$.

As $\theta$ goes to $0$, the characteristic function of $C_\theta$ converges pointwise (outside of a set of measure $0$) to that of the symmetric slab of width $2h$, while as $\theta \to \frac{\pi}{2}$, it converges to the characteristic function of the half-space $\{(x, y): x \le -w\}$ (see Figures \ref{fig:h} and \ref{fig:v}, respectively). Since each of these sets has Gaussian measure $p$, by dominated convergence, we have the following:

\begin{fact}
\label{supereasyfact}
Let $m$ be as above. Then $\lim_{\theta\to 0} m(\theta) = \lim_{\theta\to \frac{\pi}{2}} m(\theta) =\frac{p}{2}$.
\end{fact}

\begin{figure}[ht]
\centering
\includegraphics[width=0.85 \textwidth]{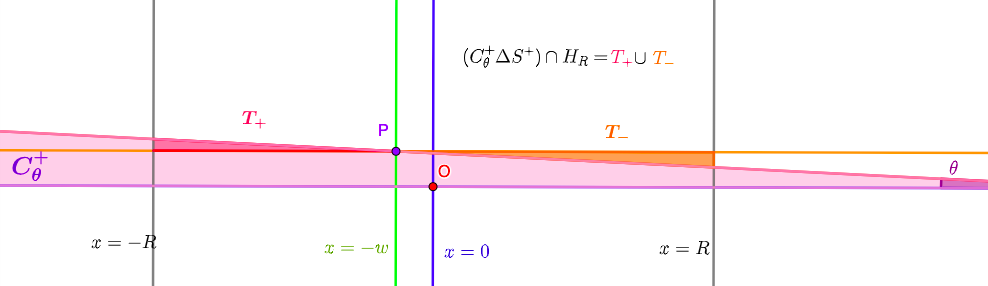}
\caption{Convergence of $C$ to a horizontal strip as $\theta\to 0$}
\label{fig:h}
\end{figure}

\begin{figure}[ht]
\centering
\includegraphics[width=0.6 \textwidth]{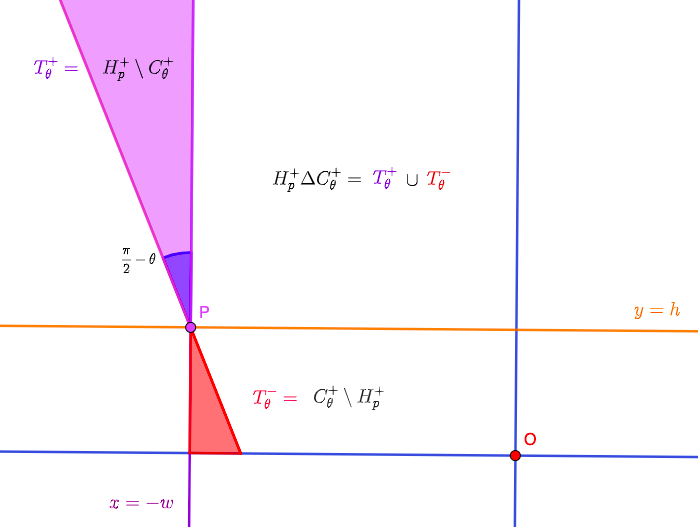}
\caption{Convergence of $C$ to a half-space as $\theta\to \frac{\pi}{2}$}
\label{fig:v}
\end{figure}

We now compute $m'(\theta)$, for which we introduce some more notations. Denote by $\theta_0=\arctan\left(\frac{h}{w}\right)$, the angle between the $x$-axis and $(PO)$; i.e., $\theta_0$ is the angle such that the apex of $C$ lies at the origin. Denote $c=\cos(\theta)$, $s=\sin(\theta)$, $h_y=ys$, $w_y=\frac{h}{s}-yc=\frac{ys}{\tan(\theta_0-\theta)}$, $u=yc=u_{\theta}$ (since $y=y_{\theta}=\frac{h}{\tan(\theta)}-w$). See Figure \ref{mainFigure} for the geometric meaning of $h_y$ and $w_y$. Note that as $\theta \to 0$, $w_y \to w$ and $h_y \to h$; while as $\theta \to \frac{\pi}{2}$, $h_y\to -w$ and $w_y\to h$.

Also denote by $N\sim \mathcal{N}(0,1)$ a standard gaussian, and let $\lambda_{\theta}:=\E(N|N>-u)$ (with $u=yc$); integrating by parts, one sees that $\lambda_\theta = \frac{\phi(u)}{\Phi(u)}$, where $\phi(t)= \frac{e^{-t^2/2}}{\sqrt{2\pi}}$ is the density of of a (real) Gaussian, and let $\Phi(t)=\gamma_1((-\infty, t])$ denote the Ehrhard function, as before.

\begin{figure}[ht]
\centering
 \includegraphics[width=0.9 \textwidth]{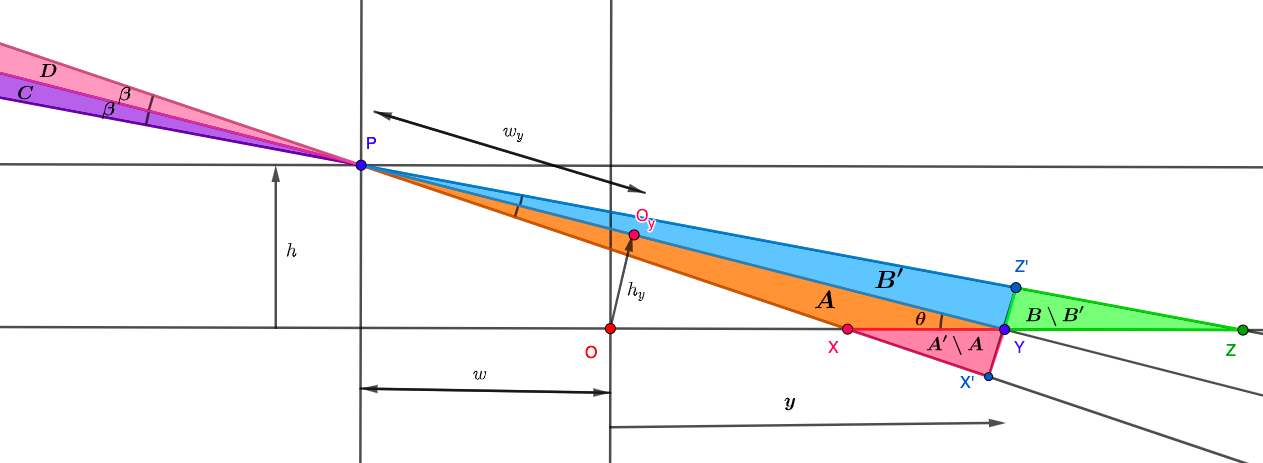} 
 \caption{Notations for the proof of Lemmas \ref{mprimesign} and \ref{convexaupointcritique}}
  \label{mainFigure}
\end{figure}
\begin{lem}
\label{mprimesign}
Let $p\in (0,1)$. Then for any $\theta\in [0,\frac{\pi}{2}]$, \begin{equation}\label{eq:mpr}
    m'(\theta)=\PP(N\leq yc) \frac{e^{-h_y^2/2}}{\sqrt{2\pi}} (\lambda_{\theta} - w_y)
\end{equation} 
In particular, $m'(\theta)$ has the same sign as $\lambda_{\theta}-w_y$.
\end{lem}

\begin{proof}
Denote $y'=||PY||=\frac{h}{\sin(\theta)}$. For $\beta>0$ a small parameter, denote $\theta_x:=\theta+\beta$ and $\theta_z:=\theta-\beta$. Denote $-w<x<y<z$ the corresponding coordinates, i.e. $\tan(\theta_x)=\frac{h}{w+x}$, $\tan(\theta_z)=\frac{h}{w+z}$. Denote $A$ the triangle $PXY$, $B$ the triangle $PYZ$, and $C$ and $D$ the (infinite) cones with apex $P$, angle $\beta$, and base $PY$ opening to the left, as on Figure \ref{mainFigure}.
Denote $\sigma_y$ the axial symmetry with axis $(PY)$, so that $D=\sigma_y(C)$. Then

$$m'(\theta)=\lim_{\beta\to 0} \frac{m(\theta+\beta)-m(\theta-\beta)}{2\beta}=\lim_{\beta\to 0} \frac{\gamma(C_x)-\gamma(C_z)}{2\beta}= 2\lim_{\beta\to 0} \frac{\gamma(C)+\gamma(D)-(\gamma(A)+\gamma(B))}{2\beta}.$$

Let $\Delta'$ be the perpendicular to $(PY)$ passing through $Y$. Then $\Delta'$ divides the triangle $B$ into two pieces, we denote $B'$ the one containing $P$ (so that, $Y$ being fixed, $B\setminus B'$ has area of order $\beta^2$). Let $A'=\sigma_y(B')$ be the triangle symmetric to $B'$ (it contains $A$). Denote $\alpha=\tan(\beta)$.  Then
\begin{align*}
\gamma(C)+\gamma(D) &= \frac{1}{2\pi}\int_0^{+\infty} e^{-(x+w_y)^2/2} \int_{-\alpha x}^{\alpha x} e^{-(h_y+ t)^2/2} \,dt\,dx \\
&= 2\alpha e^{-h_y^2/2}   \int_0^{+\infty} x e^{-(x+w_y)^2/2}\,\frac{dx}{2\pi} +O(\alpha^2) \\
&= 2\alpha\left( \frac{e^{-(h^2+w^2)/2} }{2\pi} - w_y \frac{e^{-h_y^2/2}}{\sqrt{2\pi}} \PP(N>w_y)\right) + O(\alpha^2).
\end{align*}
where we have used that $h_y^2 + w_y^2 = h^2 + w^2$.

On the other hand, $\gamma(A)+\gamma(B)=\gamma(A')+\gamma(B')+(\gamma(B\setminus B')-\gamma(A'\setminus A))=\gamma(A')+\gamma(B')+O(\beta^2),$ 
and (recall $y'=||PY||=yc+w_y$):
\begin{align*}
\gamma(A')+\gamma(B')= \alpha \int_0^{y'} e^{-(x-w_y)^2/2} \int_{- x}^{x} e^{-(h_y+\alpha t)^2/2} \,dt\,\frac{dx}{2\pi} = 2\alpha e^{-h_y^2/2} \int_0^{y'} xe^{-(x-w_y)^2/2}\,\frac{dx}{2\pi} + O(\alpha^2)
\end{align*}

Observe that
\begin{align*}\int_0^{y'} xe^{-(x-w_y)^2/2} \frac{dx}{2\pi} &=\frac{e^{-w_y^2/2}}{2\pi}+w_y \frac{\PP(N\leq w_y)}{\sqrt{2\pi}} -\int_{y'}^{+\infty} xe^{-(x-w_y)^2/2} \\
&=\frac{e^{-w_y^2/2}}{2\pi}-\frac{e^{-(yc)^2/2}}{2\pi}+w_y \frac{\PP(N\leq w_y)-\PP(N\geq yc)}{\sqrt{2\pi}}.
\end{align*}

Hence, subtracting and simplifying (again using $h_y^2 + w_y^2 = h^2 + w^2$, as well as $h_y^2 + (yc)^2 = y^2$), one gets
$$\lim_{\beta \to 0}\frac{\gamma(C)+\gamma(D)-\gamma(A)-\gamma(B)}{2\beta}=  \frac{ e^{-y^2/2}}{2\pi} - w_y e^{-h_y^2/2} \frac{\PP(N\geq -yc)}{\sqrt{2\pi}}.$$
Factoring out $e^{-h_y^2/2}=e^{-(ys)^2/2}$ and $\PP(N\leq yc)=\PP(N\geq -yc)$, one finds
$$m'(\theta)=\PP(N\leq yc) \frac{e^{-h_y^2/2}}{\sqrt{2\pi}} \left(\E(N | N\geq -yc) - w_y \right) =:\PP(N\leq yc) \frac{e^{-h_y^2/2}}{\sqrt{2\pi}} ( \lambda_{\theta} - w_y).$$
\end{proof}

%%%%
We remark that this derivative formula holds for any $p\in (0,1)$ (even $p > \frac{\pi}{2}$) and any $y>-w=\phi^{-1}(p)$ (i.e., for any $\theta\in (0,\frac{\pi}{2})$). See Figures \ref{fig:largeplargetheta} and \ref{fig:largepsmalltheta} for the geometric interpretation of the variable $w_y$ when $p > \frac{1}{2}$.
 
\begin{figure}[ht]
\centering
\includegraphics[width=0.67\textwidth]{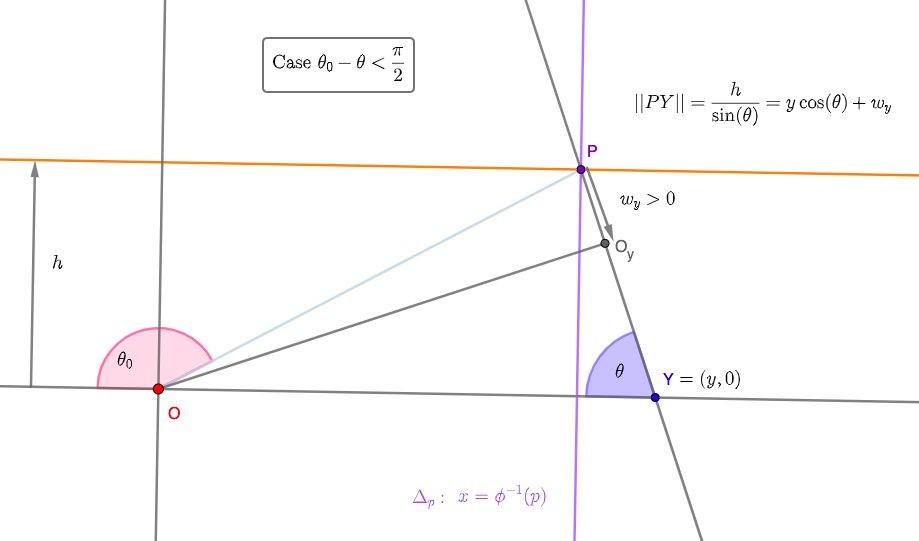} 
\caption{Geometric interpretation of $w_y$ when $p > \frac{1}{2}$ and $\theta$ is large} \label{fig:largeplargetheta}
\end{figure}

\begin{figure}[ht]
\centering
\includegraphics[width=0.67 \textwidth]{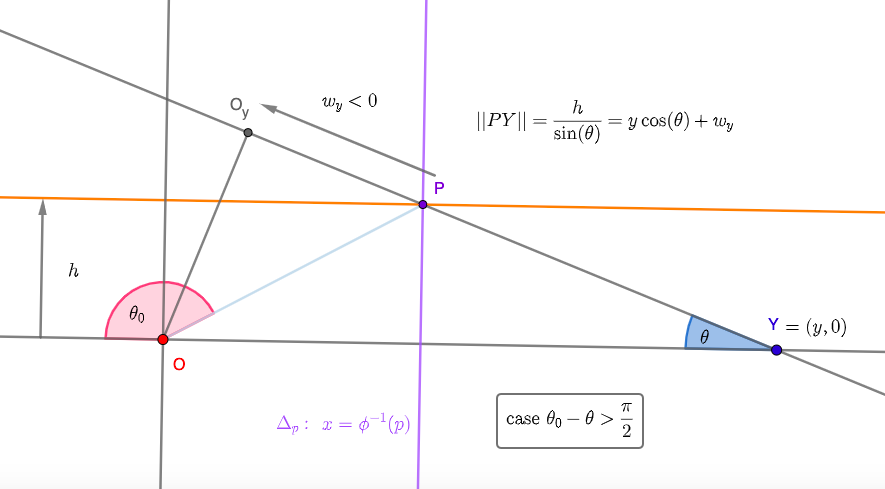} 
\caption{Geometric interpretation of $w_y$ when $p > \frac{1}{2}$ and $\theta$ is small} \label{fig:largepsmalltheta}
\end{figure}

Using \eqref{eq:mpr}, one easily sees $m'(\theta) \to -\frac{w e^{-h^2/2}}{\sqrt{2\pi}}$, which for $p < \frac{1}{2}$ yields that $m(\theta)$ is decreasing near $0$. (This actually holds true for $p = \frac{1}{2}$ as well -- indeed, in that case $m'(0) = 0$ but $m''(0) < 0$. However, for $p > \frac{1}{2}$, $m(\theta)$ is increasing near $0$.) When $\theta\to \frac{\pi}{2}$, $m'(\theta)\to \frac{e^{-w^2/2}}{\sqrt{8\pi}} \left(\sqrt{\frac{\pi}{2}} - h\right) > 0$ (since $p\leq \frac{1}{2} < \psi\left(\sqrt{\frac{\pi}{2}} \right)$). Hence $m$ is increasing near $\frac{\pi}{2}$. Therefore $m$ has at least one local minimum, which is the global minimum of $m$ on $(0,\frac{\pi}{2})$. The next two lemmas imply that $m$ has no other critical point, and in particular, no local maximum, so that $\max_{[0,\frac{\pi}{2}]}m=\frac{p}{2}$ is only attained at $0$ and at $\frac{\pi}{2}$, as needed.

Recall that at $\theta = \theta_0$, $C(\theta_0)$ has apex at the origin. The next lemma claims that the function $(\theta \mapsto \gamma_2(C_{\theta}^+))$ is convex for $\theta > \theta_0$ (corresponding to $y\leq 0$).
\begin{lem}
\label{nicelemma}
The function $m$ is convex on $[\theta_0, \frac{\pi}{2})$.
\end{lem}
\begin{proof}
First assume $\theta_0 < \theta < \frac{\pi}{2}$, and let $0>y>-w$ be such that $\theta=\theta_y$. Let $0<\beta<\min\{\theta, \frac{\pi}{2}-\theta\}$. Let $z>y>x>-w$ so that $\theta_x=\theta +\beta$ and $\theta_z=\theta-\beta$. We aim at showing that $m(\theta_x)+m(\theta_z)>2m(\theta_y)$, i.e that $\gamma(C_x^+)+\gamma(C_z^+)>2\gamma(C_y^+)$. Denote $A, B, C, D$ the two triangles and two cones as in Figure \ref{fig:convex}: $A$ is the triangle $XYP$, $B$ is the triangle $YZP$, while $C$ is the cone of angle $\beta$ (not containing $B$) bordered by $(PY)$ and $(PZ)$, and  $D$ is the cone of angle $\beta$ (not containing $A$) bordered by $(PY)$ and $(PX)$.
Denote by $\sigma_y$ the symmetry with axis $(PY)$, so that $D=\sigma_y(C)$, and let $B':=\sigma_y(A) \subsetneq B$. We have 
$$\gamma(C_x^+)+\gamma(C_z^+)-2\gamma(C_y^+)=\gamma(B)-\gamma(A)+\gamma(D)-\gamma(C).$$

Since $D=\sigma_y(C)$ and $C$ is on the side of $(PY)$ further from the origin, we have $\gamma(D)>\gamma(C)$. Similarly, since $B'=\sigma_y(A)$, and $B'$ is on the side of $(PY)$ closer to the origin, we have $\gamma(B')>\gamma(A)$. Also, $\gamma(B)>\gamma(B')$ since $B$ (strictly) contains $B'$. It follows that $m(\theta_x)+m(\theta_z)>2m(\theta_y)$, i.e. that $m$ is (strictly) convex at $\theta$.

\begin{figure}
\centering
\includegraphics[width=0.67 \textwidth]{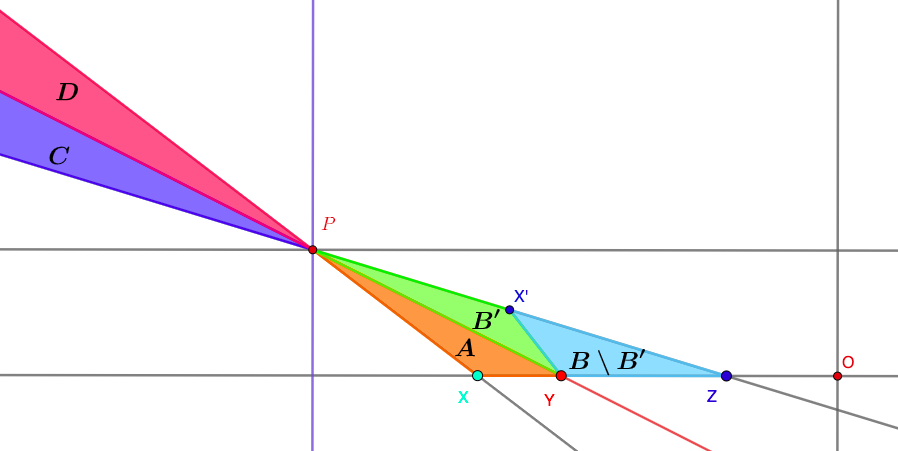} 
\caption{Notations for the proof of Lemma \ref{nicelemma}}
\label{fig:convex}
\end{figure}

In the case $\theta = \theta_0$, the drawing is the same; however, in this case we have $Y = 0$, so $C$ and $D$, as well as $A$ and $B'$ are now symmetric with respect to a line passing through the origin and hence have the same measure. However, $B'$ is still strictly contained in $B$ and so we have
$$m(\theta_x) + m(\theta_z) - 2m(\theta_0) = (\gamma(C) - \gamma(D)) + (\gamma(B') - \gamma(A)) + \gamma(B \backslash B')= \gamma(B \backslash B')> 0.$$
\end{proof}

\begin{lem}
\label{convexaupointcritique}
Assume $0<\theta< \theta_0$ and $m'(\theta)=0$. Then $m''(\theta)>0$.
\end{lem}

Lemma \ref{mprimesign} shows that $m'(\theta)=g(\theta)(\lambda_{\theta}-w_y)$, for some $g$ smooth and positive on $(0,\frac{\pi}{2})$. Therefore Lemma \ref{convexaupointcritique} is equivalent to the claim that, if $m'(\theta)=0$ and $\theta<\theta_0$, then $\frac{d}{d\theta}(\lambda_{\theta}-w_y)>0$. Recall that $\lambda_{\theta}=\frac{\phi(u)}{\Phi(u)}$ for $u=u_{\theta}=yc$, and $\phi = \frac{d\gamma_n}{dx}$. In other words, to derive Lemma \ref{convexaupointcritique} it suffices to show that $\frac{d\lambda_{\theta}}{du}\frac{du}{d\theta} > \frac{d w_y}{d\theta}$, when $\lambda_{\theta}=w_y$ and $y>0$.

One easily computes that $\frac{\partial}{\partial u} \lambda_{\theta}=-u \lambda_{\theta} -\lambda_{\theta}^2=\text{Var}_{\theta}-1 \in (-1,0)$, where $\text{Var}_{\theta}$ denotes the conditional variance 
$$\text{Var}_{\theta}=\E(N^2 | N>-u) -(\E(N | N>-u) )^2=\E(N^2 | N>-u)- \lambda_{\theta}^2.$$

Denote $x:=y+w$, so that $ys = xs-ws=hc-ws$. We see from Figure \ref{mainFigure} that $w_y^2=h^2+w^2-(ys)^2=h^2+w^2-(hc-ws)^2=(hs+wc)^2$, and hence $w_y=hs+wc$. Recalling that $c = \cos\theta, s = \sin\theta$, this yields $\frac{dw_y}{d\theta} = hc-ws=ys$. Now $u=yc=y'-w_y = \frac{h}{s}-w_y$ and so $\frac{du}{d \theta}=\frac{dy'}{d\theta}-\frac{dw_y}{d \theta}=-\frac{y'}{t}-ys$. Altogether, one finds that, if $m'(\theta)=0$, i.e. if $\lambda_{\theta}=w_y$, then
$$m''(\theta)>0 \Leftrightarrow \frac{d\lambda_{\theta}}{du}\frac{du}{d\theta} > \frac{dw_y}{d\theta} \Leftrightarrow 
 \left(\frac{y'}{t}+ys \right)(u\lambda_{\theta}+\lambda_{\theta}^2) > ys
\Leftrightarrow 
 \left(\frac{y'}{t}+ys \right)(u w_y+w_y^2) > ys.$$

 We rewrite this inequality in terms of $h$, $w$, and $x = y + w$ as follows:
 \begin{equation}\label{eq:lem5_ineq}
     (h^2+wx) (x^3+2xh^2-wh^2) > h^2 y.
 \end{equation} 
 Indeed, 
 $$\frac{ys}{ys+\frac{y'}{t}}=\frac{h^2y}{h^2y+\frac{x^3}{c^2}},$$
 and 
 $$h^2y+\frac{x^3}{c^2}=h^2y+x^3(1+t^2)=h^2y+h^2x+x^3=x^3+2xh^2-wh^2,$$ 
 while
 $$u w_y+w_y^2=(u+w_y) w_y=y' w_y=\frac{h}{s}(hs+wc)=h^2+w \frac{h}{t}=h^2+wx$$

Recapitulating, it only remains to show \eqref{eq:lem5_ineq}: that is, for any $p\leq \frac{1}{2}$, $(h,w):=(\Psi^{-1}(p),-\Phi^{-1}(p))$ and any $x > w$ (as $x = y + w$ and $y > 0$ because $\theta < \theta_0$), one has that
$$(h^2+wx) (x^3+2xh^2-wh^2) > h^2(x-w).$$

Since $x > w\geq 0$ and since $x^3+2xh^2-wh^2\geq  x^3+2xh^2-2wh^2\geq (x-w)(x^2+2h^2)$, it suffices to show that $(h^2+w^2) (x^2+2h^2)  \geq h^2$. We prove this in two steps (Claims \ref{claimeasy} and \ref{penible}):

\begin{claim}
\label{claimeasy}
Assume $y>0$ is such that $m'(\theta)=0$, with $\theta=\theta_y$. Then $(w+y)^2+h^2=(y')^2> \frac{2}{\pi}$.
\end{claim}
\begin{proof}
Let $y>0$ be such that $(y')^2\leq  \frac{2}{\pi}$. Then $w_y=y'-yc=y'-u \leq \sqrt{\frac{2}{\pi}}-u$, while $\phi(u)=\int_{-\infty}^u e^{-t^2/2} dt=\sqrt{\frac{\pi}{2}}+\int_0^u e^{-t^2/2}dt <\sqrt{\frac{\pi}{2}}+u$ (where $u=yc>0$). Therefore
$$\frac{w_y}{\lambda_{\theta}}=w_y \phi(u) e^{u^2/2}<\left( \sqrt{\frac{2}{\pi}}-u\right)\left(\sqrt{\frac{\pi}{2}}+u\right) e^{u^2/2}<(1-u^2) e^{u^2/2} <1.$$
Hence $(y')^2\leq  \frac{2}{\pi}$ cannot occur at some $y>0$ corresponding to a critical point of $m$.
\end{proof}
Hence we only need to show that $(h^2+w^2) (\frac{2}{\pi}+h^2)  \geq h^2$. Therefore the next claim concludes the proof of lemma
 \ref{convexaupointcritique}.
\begin{claim}
\label{penible}
Let $p\leq \frac{1}{2}$. Denote $h=\psi^{-1}(p)$ and $w=\phi^{-1}(1-p)$. Then $(h^2+w^2) (\frac{2}{\pi}+h^2)   \geq h^2$.
\end{claim}
\begin{proof}
Since  $\gamma_1((-\infty,-w))=p$, we know that $\gamma_1([0,w))=\frac{1}{\sqrt{2\pi}} \int_0^w e^{-t^2/2}dt=\frac{1}{2}-p$, while $\frac{1}{\sqrt{2\pi}} \int_0^w e^{-t^2/2}dt=\gamma_1([0,h])=\frac{p}{2}$. We deduce that $\int_0^w e^{-t^2/2}dt +2 \int_0^h e^{-t^2/2}dt=\sqrt{\frac{\pi}{2}}$.

We shall use the estimate $\frac{50}{16}<\pi < \frac{22}{7}$. We have $\int_0^w e^{-t^2/2}dt +2 \int_0^h e^{-t^2/2}dt=\sqrt{\frac{\pi}{2}}> \frac{5}{4}$. Since $\frac{2}{\pi}>\frac{7}{11}$, it suffices to show that $(1+\frac{w^2}{h^2}) (\frac{7}{11}+h^2)   \geq 1$, i.e. that $h^2+w^2+\frac{7}{11} \frac{w^2}{h^2} \geq \frac{4}{11}$.
We consider three distinct cases, and show the desired inequality in each case.

\underline{Case $1$}: $h^2\leq \frac{5}{22}$, and hence $h^2+w^2+\frac{7}{11} \frac{w^2}{h^2} \geq h^2+ \frac{19}{5} w^2$.
Using Cauchy-Schwarz inequality:
$$\left(h^2+\frac{19}{5} w^2\right)\left(4+\frac{5}{19}\right) \geq (w+2h)^2 \geq \left(\int_0^w e^{-t^2/2}dt +2 \int_0^h e^{-t^2/2}dt\right)^2 > \frac{25}{16}.$$
Therefore $h^2+w^2+\frac{7}{11} \frac{w^2}{h^2} \geq h^2+ \frac{19}{5} w^2 > \frac{25}{16} \frac{19}{81} > \frac{4}{11}$, as needed.

\vspace{2mm}
\underline{Case $2$}: $h^2 \geq \frac{4}{11}$.   Then the inequality $h^2+w^2+\frac{7}{11} \frac{w^2}{h^2} \geq \frac{4}{11}$ holds.

\vspace{2mm}
\underline{Case $3$}: $\frac{5}{22} <h^2 < \frac{4}{11}$.   Then  $h^2+w^2+\frac{7}{11} \frac{w^2}{h^2} > h^2 +\frac{11}{4} w^2$.
We shall use as in case $1$ that $\int_0^w e^{-t^2/2} dt \leq w$, but we shall upper bound $\int_0^h e^{-t^2/2} dt$ more carefully.
Note that, if $z\in [0,\frac{1}{2}]$, one has $e^{-z}\leq 1- z+\frac{z^2}{2}=1-z(1-\frac{z}{2}) \leq 1-\frac{3}{4} z$. Thus
$$\int_0^h e^{-t^2/2} dt \leq \int_0^h (1-\frac{3}{8} t^2)dt=h- \frac{h^3}{8} \hspace{1mm}\text{,  therefore } \hspace{1mm}w+2h \geq \int_0^w e^{-t^2/2} dt +2  \int_0^h e^{-t^2/2} dt +\frac{h^3}{4}> \frac{5}{4}+\frac{h^3}{4}.$$
Using agin Cauchy-Schwarz inequality:
$$\left(h^2+\frac{11}{4} w^2\right)\left(4+\frac{4}{11}\right) \geq (w+2h)^2 > \left( \frac{5}{4}+\frac{h^3}{4}\right)^2.$$
Since $h^2 >\frac{5}{22}>\frac{4}{25}$, we have $h^3=h^2. h> \frac{5}{22} . \frac{2}{5}=\frac{1}{11}$.
If we had $h^2 +\frac{11}{4} w^2 \leq \frac{4}{11}$, then we would have
$$\frac{4}{11}. \frac{49}{11}>\frac{4}{11}\left(4+\frac{4}{11}\right) \geq \left(h^2+\frac{11}{4} w^2\right)\left(4+\frac{4}{11}\right) >  \left( \frac{5}{4}+\frac{h^3}{4}\right)^2$$
from which one deduces $ \frac{5}{4}+\frac{h^3}{4} <\frac{14}{11}$, i.e. $h^3< \frac{1}{11}$. Hence $h^2 +\frac{11}{4} w^2 > \frac{4}{11}$, as desired.
\end{proof}
%%%%%%%
Lemmas \ref{nicelemma} and \ref{convexaupointcritique} allow to conclude that $m$ admits a unique critical point on $(0,\frac{\pi}{2})$, which is a minimum. Hence $m(\theta) \leq \frac{p}{2}=m(0)=m(\frac{\pi}{2})$, for all $\theta\in [0,\frac{\pi}{2}]$, and the inequality is strict if $0<\theta <\frac{\pi}{2}$. Therefore Lemma \ref{lem:red2}, and hence Proposition \ref{prop:planar} are proven.

\section{Remarks and open questions}
\label{section:remarks}
\subsection{The regime  \texorpdfstring{$p > \frac{1}{2}$}{p > ½}}\label{subsec:largep}
Recall from the previous section that for a fixed $p \in (0, 1)$ and $\theta \in (0, \frac{\pi}{2})$, we define $C_{\theta}$ to be a cone symmetric with respect to the $x$-axis, opening at an angle $2\theta$ towards the left, and such that $C_{\theta} \cap \{x = \Phi^{-1}(p)\}$ is a (symmetric) vertical segment of (gaussian) measure $p$, and that we set $m(\theta) =\gamma_2(C_\theta)$, so that $m(0)=m(\frac{\pi}{2})=p$.

The computation on page $10$ showed that $m'(\theta) \to \frac{ \Phi^{-1}(p) e^{-h_p^2/2}}{\sqrt{2\pi}}$ (with $h_p=\Psi^{-1}(p)$), so that for any $p \in (\frac{1}{2}, 1)$, $m'(0)>0$. Hence, it is no longer true that $\gamma_2(C)<p$ for any $C\in \mathcal{C}_2(p)$, and so Proposition \ref{prop:planar} cannot hold for these $p$. 

However, it is possible to show the following weaker statement in this regime:  

\begin{prop}
\label{prop:planarbis}
  Fix $p\in (\frac{1}{2},1)$. Let $W=\{(x,y)\in \R^2: x\leq t_y\}$ for some concave function $(t_y)$. Assume that $||W\cap \Delta_p||_2<2\Psi^{-1}(p)$. Then $\gamma_2(W)< q$, where $q:= \Phi(\sqrt{2} \Phi^{-1}(p)) > p$.  
\end{prop}

Using Proposition \ref{prop:planarbis} and following the induction argument developed in \S \ref{section:main}, one may obtain a bound on $\beta(B_2^n, V)$ in terms of $\gamma_n(V)$, but the bound is no longer dimensionless. Instead, one has:
\begin{prop}
\label{prop:largep_bd}
If $V\in \mathcal{K}^n$, $n \ge 2$, has measure $\gamma_n(V)\geq p>\frac{1}{2}$, then $\alpha(B_2^n, V)\leq f_n(p):=(2\Psi^{-1}(p_n))^{-1}$, with $p_n=\Phi(2^{-n/2} \Phi^{-1}(p))> \frac{1}{2}$.
\end{prop}

However, for any $p>\frac{1}{2}$, $\lim_{n \to \infty} f_n(p)=c=(2\Psi^{-1}(1/2))^{-1}$, so that Proposition \ref{prop:planarbis} does not yield an improvement to Theorem \ref{thm:main}.

Moreover, for $n\geq 5$, a much simpler idea yields a sharper inequality than Proposition \ref{prop:largep_bd}. Indeed, observe that if $V$ is a convex set with $\gamma_n(V)\geq p >\frac{1}{2}$, then $V$ must contain a ball of radius $r_p = \Phi^{-1}(p)$ (since, otherwise $V$ would be contained in a half-space at distance at most $r_p$ from the origin, and so $\gamma_n(V) < p$). In other words, $r_p B_2^n:=\Phi^{-1}(p) B_2^n \subset V$, and therefore $\alpha(B_2^n, V)\leq \alpha(B_2^n, r_p B_2^n) = \frac{\sqrt{n}}{2\Phi^{-1}(p)}$. If $n\geq 5$, then for $p$ close enough to $1$, one may check that $\frac{\sqrt{n}}{\Phi^{-1}(p)}<f_n(p)=\frac{1}{2\Psi^{-1}(p_n)}$. 

\subsection{Implications for \texorpdfstring{$\beta(B_2^n, B_n^\infty)$}{β(B₂, B∞)}} \label{subsec:smallp}

One may wonder if the resolution of Conjecture \ref{BSconj} has any implications concerning $\beta(B_2^n, B_{\infty}^n)$. Unfortunately, at least in the regime $p\in (0,\frac{1}{2}]$, it does not.
To see this, denote
$$g_{\beta}(p):=\limsup_{n\geq 1} \sup \{\beta(B_2^n, V): \gamma_n(V)\geq p, V\in \mathcal{C}_n\},$$
and let $S=\R^{n-1} \times [-\psi^{-1}(p),\psi^{-1}(p)]$ be a slab of gaussian measure $p\leq \frac{1}{2}$. Then $g_{\beta}(p)\geq \beta(B_2^n , S)\geq \alpha(B_2^n, S)=\frac{1}{2\psi^{-1}(p)}$. Hence if one chooses $t=t_{p,n}=\psi^{-1}(p^{1/n})$ so that $\gamma_n(tB_{\infty}^n)=p$, one gets  $\beta(B_2^n, B_{\infty}^n)= t\beta(B_2^n, tB_{\infty}^n)\leq t g_{\beta}(p)$. But, up to a constant, the best upper bound one gets through this estimate (in the range $p \in (0, \frac{1}{2}]$) is obtained at $p = \frac{1}{2}$:
$$\inf_{p \in (0, \frac{1}{2}]} t_{p,n} g_{\beta}(p)= \inf_{p \in (0, \frac{1}{2}]} \psi^{-1}(p^{1/n}) g_{\beta}(p) \geq \inf_{p \in [0, \frac{1}{2})} \frac{\psi^{-1}(p^{1/n})}{2\psi^{-1}(p)}$$
and it turns out that this infimum is $\Omega(\sqrt{\log n})$ (see below).
However, just considering $p = \frac{1}{2}$, one has $g_{\beta}(1/2)t_{1/2, n} \leq g_{\beta}(1/2) \sqrt{2 \log (n)}\leq 5 \sqrt{2\log n}$ using Banaszczyk's inequality (\ref{Banresult}).

\begin{proof}[Proof that $\inf_{0<p<\frac{1}{2}} t_{p,n} g_{\beta}(p)=\Omega(\sqrt{\log n})$] Since $\Psi^{-1}$ is convex on $(0,1)$, and $\Psi^{-1}(0)=0$, one has $\Psi^{-1}(\lambda q) \leq \lambda \Psi^{-1}(q)$, for any $\lambda, q\in (0,1)$. If $p\in \left(0, (\log n)^{-\frac{n}{2(n-1)}}\right)$, then, denoting $\lambda=p^{1-\frac{1}{n}} \in (0,1)$, convexity yields  $\Psi^{-1}(p)=\Psi^{-1}(\lambda p^{1/n}) \leq \lambda \Psi^{-1}(p^{1/n})$, and hence $\frac{\Psi^{-1}(p^{1/n})}{\Psi^{-1}(p)} \geq \lambda^{-1}=p^{-\frac{n-1}{n}} \geq \sqrt{\log n}$.

Denote $\alpha(x)=\frac{e^{-x^2/2}}{x\sqrt{2\pi}}\left(1-\frac{1}{x^2}\right)$. One verifies that $x \mapsto \alpha(x)$ is decreasing for $x\geq 3^{1/4}$. Integration by parts yields the standard estimate (for all $x>0$) for the tail probability of a Gaussian:
$$\alpha(x) = \frac{e^{-x^2/2}}{x\sqrt{2\pi}}\left(1-\frac{1}{x^2}\right) \leq 1 - \Phi(x) \leq \frac{e^{-x^2/2}}{x\sqrt{2\pi}}$$
from which we deduce that $\Psi(x) = 2\Phi(x) - 1 \leq 1-2 \alpha(x)$. Let $\alpha_0 = \alpha(3^{1/4})$ and denote by $x \mapsto x_{\alpha}$ the inverse function of $\alpha$ on $(0,\alpha_0]$. For any $\alpha \leq \alpha_0$, one has $\Psi^{-1}(1-2\alpha) \geq x_{\alpha}$.

Assume $p\in  [(\log n)^{-\frac{n}{2(n-1)}},\frac{1}{2}]$. Then $p^{1/n} \geq (\log n)^{-\frac{1}{2(n-1)}}\geq 1- \frac{\log \log n}{2(n-1)}=: 1-2x$, and hence 
$$\Psi^{-1}(p^{1/n}) \geq \Psi^{-1}(1-2x) \geq \sqrt{\log n}$$
as long as $n$ is large enough so that $\frac{\log \log n}{4(n-1)}\leq \frac{1}{\sqrt{2\pi}\sqrt{n\log n}}\left(1-\frac{1}{\log n}\right)$. Therefore, for large enough $n$, for any $p\in  [(\log n)^{-\frac{n}{2(n-1)}},\frac{1}{2}]$, one gets
$$\frac{\Psi^{-1}(p^{1/n})}{\Psi^{-1}(p)} \geq \frac{\sqrt{\log n}}{\Psi^{-1}(1/2)} > \sqrt{\log n}.$$
\end{proof}

\subsection{The limiting behavior as \texorpdfstring{$p \to 1$}{p → 1}}
\label{subsec:pto1}

Set 
$$g_{\alpha}(n,p)= \sup \{\alpha(B_2^n, V) : V\in \mathcal{C}^n, \gamma_n(V) = p\}$$
and similarly 
$$g_\beta(n,p)=\sup \{\beta(B_2^n, V) : V\in \mathcal{C}^n, \gamma_n(V) = p\}.$$
For $\alpha$, one may also consider the optimal bounding function for non-symmetric $V$:
$$g_{\alpha}^{ns}(n,p)= \sup \{\alpha(B_2^n, V) : V\in \mathcal{K}^n, \gamma_n(V) = p\}$$
We could also define $g_{\beta}^{ns}(n,p)$ in a similar manner, by taking a supremum over the $V\in \mathcal{K}_0^n$ of measure $\gamma_n(V)=p$. However, as we saw in section \ref{section:BSconj}, $\beta(B_2^n, V)$ behaves badly for non-symmetric $V$: we have seen that $g_{\beta}^{ns}(n,p)=+\infty$ for any $p<\frac{1}{2}$. By definition, one has $g_\alpha \le g_\alpha^{ns}$, and $g_\alpha \le g_\beta$ by Fact \ref{alphabetaineq}. Theorem \ref{thm:main} implies that $g_\alpha=g_\alpha^{ns}$ on $(0,\frac{1}{2}]$ (recall that, for $n\geq 2$, $g_{\alpha}(n,p)=\alpha(B_2^n, S_p)=(2\Psi^{-1}(p))^{-1}$ for $0<p \leq \frac{1}{2}$, with $S_p$ the symmetric slab $S_p=\R^{n-1} \times [-\Psi^{-1}(p),\Psi^{-1}(p)]$), while it is unknown whether equality holds in the $p>\frac{1}{2}$ regime.

It is easy to see that, for any fixed $p\in (0,1)$, $g_{\alpha}(n,p)$ is non-decreasing in $n$. Indeed, let $V\subset \R^n$ is an arbitrary closed convex set with $\gamma_n(V)=p$, and consider $V'=V\times \R \subset \R^{n+1}$ (which has measure $\gamma_{n+1}(V')=\gamma_n(V)=p$). Then
\begin{align*}
    g_{\alpha}(n+1,p) \geq \alpha(B_2^{n+1}, V') &=\sup \{\mu(L', V') \big| \text{Span}(L'\cap B_2^{n+1})=\R^{n+1} , L' \text{ $(n+1)$-lattice}\} \\
& \geq  \sup \{\mu(L+\Z e_n, V') :  \text{Span}(L\cap B_2^{n})=\R^{n} , L \text{ $n$-lattice}\} 
= \alpha(B_2^n, V).
\end{align*}
The same tensorization argument shows that $g_\beta(n,p)$ is non-decreasing in $n$, for any $p\in (0,1)$.

Proposition \ref{prop:largep_bd}, or the inclusion $\Phi^{-1}(p) B_2^n \subset V$ given above, shows that $\lim_{p\to 1} g_{\alpha}^{ns}(n,p)=0$, for any fixed $n$; the same argument shows that the same is true for $g_\alpha, g_\beta$.  A natural question is then whether one also has $\lim_{p\to 1} \sup_n f(n,p)=0$ for $f \in \{g_\alpha, g_\beta, g_\alpha^{ns}\}$. As it turns out, this is not the case:
\begin{prop}
\label{limitnonzero} 
$$\lim_{p\to 1} \lim_{n\to +\infty} g_{\beta}(n,p) \geq 1 \hspace{2mm}\text{ and }\lim_{p\to 1} \lim_{n\to +\infty} g_{\alpha}(n,p)=\inf_{p\in (0,1)} \sup_{n\geq 2} g_{\alpha}(n,p) \geq \frac{1}{2}.$$
\end{prop}

\begin{proof}
Let $R_p=R_p(n)>0$ be such that $\gamma_n(R_p B_2^n)=p$. In other words, if $Z \sim \mathcal{N}(0,I_n)$ is a standard gaussian vector, then $R_p$ is defined by $\PP(\|Z\| \leq R_p)=p$. By standard concentration results, one has $\mathbb P(\|Z\|^2 \ge n+2\sqrt{nx}+2x) \le e^{-x}$ for any $x>0$, which implies that if $n\geq |\ln(1-p)|$, then $R_p \leq (n+4 (n |\ln(1-p)|)^{1/2} \leq \sqrt{n}+2 \sqrt{|\ln(1-p)|}$. Hence, for such $n$,

$$\beta(B_2^n, R_p B_2^n)=\frac{1}{R_p} \beta(B_2^n, B_2^n)=\frac{\sqrt{n}}{R_p} \geq \frac{\sqrt{n}}{\sqrt{n}+2 \sqrt{|\ln(1-p)|}} \geq 1-\frac{2 \sqrt{|\ln(1-p)|}}{\sqrt{n}},$$
where we have used that $\beta(B_2^n, B_2^n)=\sqrt{n}$. Similarly, using homogeneity of $\alpha$ and the fact that $\alpha(B_2^n, B_2^n)=\frac{\sqrt{n}}{2}$, one gets the lower bound $\frac{1}{2}$ in the limit.
\end{proof}

An interesting consequence of Proposition \ref{limitnonzero} above, and of the fact that $\inf_{0<p<\frac{1}{2}} t_{p,n} g_{\alpha}(p)=\Omega(\sqrt{\log n})$ (as shown in \S \ref{subsec:smallp} above), is that $p=1/2$ is the best tuning within Conjecture \ref{BSconj} for upper bounding $\beta(B_2^n, B_{\infty}^n)$ or $\alpha(B_2^n, B_{\infty}^n)$.

\subsection{Open questions}
\label{section:open}

\begin{enumerate}
    \item Recall from paragraph \ref{subsec:pto1} that for any $n, p$, $g_{\alpha}(n,p) \le g_\alpha^{ns}(n, p)$ (by definition). Theorem \ref{thm:main} tells us that $g_{\alpha}^{ns}(n,p)=g_{\alpha}(n,p)=\frac{1}{2\Psi^{-1}(p)}$ for all $n\geq 1$ and any $p\in (0,\frac{1}{2}]$. Is there equality (in $g_{\alpha}(n,p) \le g_\alpha^{ns}(n, p)$) for $n\geq 2$ and $p\in (\frac{1}{2}, 1)$?
    
    \item  Theorem \ref{thm:main} and translation invariance of the covering radius $\mu(L, \cdot)$ (with respect to any given $n$-lattice $L$) imply that, for any convex body $V$, $\inf_x \alpha(B_2, V+x)\leq \inf_x f(\gamma_n(V+x))$ (where the infimum is over $x\in \R^n$).
with $f$ given by Theorem \ref{thm:main}.

Let $V$ be an arbitrary convex body. Does one have $\inf_x \beta(B_2^n, V+x)\leq \inf_x f_{\beta}(\gamma_n(V+x))$ for some non-increasing $f_{\beta}: (0,1) \to (0,+\infty)$ independent of $n$?
For the counterexample given in \S \ref{subsec:sym_nec}, both sides vanish, so the inequality holds. This also holds for $V=-V$, 
with $f_{\beta}$ as in Theorem \ref{thm:BSconjholds},
since $\max_x \gamma_n(V+x)=\gamma_n(V)$ for origin-symmetric $V$.

\item We would like to find the largest possible subclass $\Pi_n \subset \mathcal{K}_0^n$ such that there exists a non-increasing function $g: (0, 1] \to [0, \infty)$ independent of $n$, such that $\beta(B_2^n, V)\leq g(\gamma_n(V))$ for all $V\in \Pi_n$. By Theorem \ref{thm:BSconjholds}, one can take $\Pi_n=\mathcal{C}_n$. Can one take a larger class of convex bodies than $\mathcal C_n$? For instance, can one take $\Pi_n=\{V\in \mathcal{K}_0^n: \int_{x\in V} x d\gamma_n(x)=0 \}$? In other words, does Conjecture \ref{BSconj} hold for all convex bodies whose gaussian barycenter is at the origin?

\item  For $n\geq 2$, does there exist $c_n>0$ such that $\alpha(B_2^n, V) \geq c_n \beta(B_2^n, V)$, for any $V\in \mathcal{C}_n$? If so, is $\sup_n c_n < \infty$? If this were true, then the Komlós conjecture would be equivalent to the conjecture that $\alpha(B_2^n, B_{\infty}^n) =O_n(1)$.
\end{enumerate}

\section{Appendix}
\label{section:appendice}

The following inequality was mentioned in the introduction (and used throughout the text).

\begin{fact}
\label{alphabetaineq}
Let $U,V \subset \R^n$ be two closed convex sets. Assume $U$ is compact, $0\in U$, and $\text{int}(V)\neq \emptyset$. Then $\alpha(U,V)\leq \beta(U,V)$.
\end{fact}

If $V$ is closed convex, but not compact, note that as $U$ is compact, $\beta(U,V) = \inf \{\beta(U,K) : K\subset V \text{ compact}\}$. On the other hand,  $\alpha(U,V)\leq \inf \{\alpha(U,K) : K\subset V \text{ compact}\}$. Therefore one only needs to show that $\alpha(U, V) \le \beta(U, V)$ under the assumption that both $U, V\in \mathcal{K}^n$ are compact.

\begin{proof}
We need to show that for any $n$-lattice $L$, $\mu(L,V)\leq \beta(U,V) \lambda_n(L,U)$. By homogeneity (in $L$) of $\lambda_n$ and $\mu$, we can assume that $\lambda_n(L,U)=1$. By homogeneity (in $V$) of $\mu$ and of $\beta$, we can assume that $\beta(U,V)=1$. Hence we only need to show that if $\text{Span}(L\cap U)=\R^n$ and that $\max_{u_1, \ldots, u_n \in U}  \min_{\epsilon\in \{\pm 1\}^n} || \sum_{i=1}^n \epsilon_i u_i ||_V \leq 1$, then $L+V=\R^n$.

Fix $u_1, \ldots, u_n \in U\cap L$, linearly independent, and let $P = [0, u_1] + \cdots + [0, u_n]$ be the parallelotope they span. It is easy to see that if $P + L = \mathbb R^n$, so that $P \subset L + V$ implies $L + V = \mathbb R^n$. By Claim \ref{combiclaim} below, it suffices to show that $P^1 \subset P^0+\frac{1}{2}V$, where  $P^0=\{\sum_{i=1}^n \delta_i u_i, \delta \in \{0,1\}^n\}$, and $P^1=\{\sum_{i=1}^n \delta_i u_i, \delta \in \{0,\frac{1}{2}, 1\}^n\}$.

Any $y\in P^1$ may be written as $y=y_0 +\frac{1}{2} z_0$, with $y_0=\sum_{i\in J} u_i$, $z_0=\sum_{i\in Z} u_i$ for index sets $Z, J \subset [n]$, $Z\cap J=\emptyset$. Consider the tuple $(v_i)_{i\in [n]}$ with $v_i=u_i$ if $i\in Z$, and $v_i=0$ otherwise. As $\beta(U,V)\leq 1$, there exist $v_\pm = \sum_{z \in Z_{\pm}} z_i$, where $Z_ \cup Z_+$ is a partition of $Z$, such that $v_- - v_+ \in V$, while $v_-+v_+=z_0$. Note that $y_0 + v_+\in P^0$ (since $Z\cap J=\emptyset$). Therefore 
$$y=y_0+\frac{1}{2} (v_-+v_+)=y_0 + v_+ +\frac{1}{2} (v_--v_+) \in P^0 +\frac{1}{2} V.$$
\end{proof}

In the above proof, we used the following standard fact (e.g., \cite[Lemma 4]{Ban93})

\begin{claim}
\label{combiclaim}
Let $E$ be a normed vector space, and let $u_1, \ldots, u_t\in E$. Define $P=\{\sum_{i=1}^t \delta_i u_i, \delta \in [0,1]^t\}$. Define $P^0=\{\sum_{i=1}^t \delta_i u_i, \delta \in \{0,1\}^t\}$, and for $k\in \mathbb{N}$, $P^{k+1}=\frac{P^k+P^k}{2}$. Assume $V\subset E$ is compact and convex and that $P^1 \subset P^0 +\frac{1}{2} V$. Then $P\subset P^0 +V$.
\end{claim}

\begin{proof}
First it is immediate that $P^k=\{\sum_{i=1}^t \delta_i u_i, \delta \in [0,1]^t\cap (2^{-k} \mathbb{Z})^t \}$, and hence that $P^{k+1}=\frac{P^k+P^0}{2}$, for all $k\geq 0$. By induction on $k\geq 0$, one sees that $P^{k+1}\subset P^0+ (1-2^{-(k+1)}) V$. Indeed, pick $y\in P^{k+1}$ and rewrite $y=\frac{1}{2}(y_0+y_k)$ for some $y_0\in P^0$, $y_k \in P^k$. Then by induction, we can write $y_k=z_0+ (1-2^{-k}) v$ for some $v\in V$. Thus $y=\frac{1}{2}(y_0+z_0)+ \frac{1}{2} (1-2^{-k}) v$. By assumption, $y_1:=\frac{1}{2}(y_0+z_0) \in P^1$ can be written $y_1=y'_0+\frac{1}{2} v'$ for some $y'_0\in P^0$ and some $v'\in V$. Hence, by convexity of $V$,
$$y=\frac{1}{2}(y_0+y_k)=y'_0+\frac{1}{2} v' + \frac{1}{2} (1-2^{-k}) v \in P^0 + (1-2^{-k-1}) V.$$
Now any $u\in P$ can be written $u=\lim u_k$ with $u_k\in P^k$, and we may write $u_k=z_k+v_k$ for some $z_k\in P^0$ and some $v_k\in (1-2^{-k}) V$. By finiteness of $P^0$, and compactness of $V$, we may assume (up to passing to a subsequence) that $z_k=z_0$ is constant and that $v_k \to v\in V$. Hence $u=z_0+v \in P^0 +V$.
\end{proof}

\printbibliography

\end{document}